\documentclass[12pt, a4paper]{amsart}
\usepackage[english]{babel}
\usepackage{amsmath,enumerate, amsfonts, amssymb,amsthm, xypic}
\usepackage[dvips]{graphics} 
\usepackage{color}
\usepackage{graphicx}
\usepackage{hyperref}

\newtheorem{thm}{Theorem}[section]
\newtheorem{lemma}[thm]{Lemma}

\newtheorem{cor}[thm]{Corollary}

\theoremstyle{definition}

\newtheorem{rmk}[thm]{Remark}

\DeclareMathOperator{\R}{\mathbb R}
\DeclareMathOperator{\C}{\mathbb C}
\DeclareMathOperator{\K}{\mathbb K}

\DeclareMathOperator{\N}{\mathbb N}

\DeclareMathOperator{\alg}{alg}

\title{Artin approximation compatible with a change of variables}
\author{Goulwen Fichou, Ronan Quarez and Masahiro Shiota}
\address{IRMAR (UMR 6625), Universit\'e de Rennes 1, Campus de
  Beaulieu, 35042 Rennes Cedex, France}
  \address{IRMAR (UMR 6625), Universit\'e de Rennes 1, Campus de
  Beaulieu, 35042 Rennes Cedex, France}
\address{Graduate School of Mathematics, Nagoya University, Chikusa, Nagoya, 
464-8602, Japan}
\date\today

\begin{document}
\begin{abstract} We propose a version of the classical Artin approximation which allows to perturb the variables of the approximated solution. Namely, it is possible to approximate a formal solution of a Nash equation by a Nash solution in a compatible way with a given Nash change of variables. This results is closely related to the so-called nested Artin approximation and becomes false in the analytic setting. We provide local and global version of this approximation in real and complex geometry together with an application to the Right-Left equivalence of Nash maps.

\end{abstract}
\maketitle
Let $F(x,z)\in \K \{x,z\}$ be a convergent power series in two sets of variables $x=(x_1,\ldots,x_m)$ and $z=(z_1,\ldots,z_n)$ where $\K$ denotes the field of real numbers $\R$ or complex numbers $\C$. M. Artin \cite{A} proved in 1968 that any formal power series $\Phi \in \K[[x]]^n$ solution of $F(x,\Phi (x))=0$ can be approximated in the $\mathfrak m$-adic topology, where ${\mathfrak m}=(x)$, by convergent power series $\tilde \Phi \in \K\{x\}^n$ solution of $F(x,\tilde \Phi (x))=0$. The result remains valid replacing the analytic power series by formal power series that are algebraic over the polynomials, also called Nash power series and denoted by $\K_{\alg}[[x]]$.

There exists an important literature dealing with Artin approximation and its various generalisations (cf. \cite{R} for a recent survey). In this paper, we are interested in a version of Artin approximation involving a modification of the source variables of the function to approximate. Roughly speaking, given convergent power series $G\in \K \{x\}$ and $F\in \K\{x,y,z \}$, can one approximate formal solutions $\Phi \in \K[[x]]$ and $\Psi \in \K[[y]]$ of the equation $F(x,\Phi (x), \Psi \circ G(x))=0$ by convergent power series solutions ?

The answer to that question is negative in general, as it can be derive from Gabrielov counter-example to the nested version of Artin approximation \cite{G}. Nevertheless, we prove in Theorem \ref{thmlocal} that the answer is positive is $F$ and $G$ are of class Nash rather than only analytic. We propose moreover global versions of this approximation result in the real (Theorem \ref{thmglobalR}) and complex (Theorem \ref{thmglobalC}) algebraic setting.
A particular case of this modified version of Artin approximation has already been used by the third author, in the local case as well as in the global real case, to prove that the Right-Left analytic equivalence of real Nash maps implies their Right-Left Nash equivalence \cite{S3}. The approach proposed in the present paper enables to recover these results as direct and more conceptual consequences of Artin-type approximation.

The proof of this version of Artin approximation is based on nested approximation, which enables to approximate the solutions by keeping the dependence of the solutions in certain variables (cf. \cite{T} for example). We show in fact that the version presented in the paper is somehow equivalent to nested approximation. Note that the nested approximation, which is known to hold true in the Nash setting but not in the analytic one, is based on N\'eron desingularisation of regular ring morphisms \cite{Po,Sp}. For the global versions, we use in the real setting the work by Coste, Ruiz and Shiota \cite{CRS}, where the regular ring homomorphism considered is the inclusion morphism of Nash functions in analytic functions on a compact Nash manifold, whereas in the complex setting we used the work by Lempert \cite{L} who consider the regular ring homomorphism given by the inclusion of the constant functions in the analytic ones, in the case of compact polynomial polyhedra.

The paper is organised as follows. In the first section we deal with the simplest situation of the local case, showing that formal solutions of Nash equations can be approximated by Nash solutions. The main ingredient is nested approximation. In the second section, we tackle the global version, whose proof need some additional ingredients to achieve the approximation of analytic solutions of Nash equations by Nash solutions.

\section{Local case}
Let $\K$ denote a field.
Denote by $\K[[x_1,\ldots,x_m]]$ and $\K_{\alg}[[x_1,\ldots,x_m]]$ respectively the rings of formal powers series and algebraic power series (over the polynomials $\K[x_1,\ldots,x_m]$), endowed with the $\mathfrak m$-adic topology where $\mathfrak m$ denotes the maximal ideals generated by $(x_1,\ldots,x_m)$.

\begin{thm}\label{thmlocal} Let $l,m,n,p,q\in \N$ and $x=(x_1,\ldots,x_m)$, $y=(y_1,\ldots,y_n)$, $z=(z_1,\ldots,z_p)$ and $t=(t_1,\ldots,t_l)$. Let $F\in \K_{\alg}[[x,y,z]]^q$ and $G\in \K_{\alg}[[x]]^l$ be such that $F(0)=0$ and $G(0)=0$. 

Assume that there exists $\Phi\in \K[[x]]^n$ and $\Psi\in \K[[t]]^p$ such that $\Phi(0)=\Psi(0)=0$ and $F(x,\Phi(x), \Psi\circ G(x))=0$. Then, we can approximate $\Phi$ by $\tilde \Phi \in \K_{\alg}[[x]]^n$ and $\Psi$ by $\tilde \Psi \in \K_{\alg}[[t]]^p$ such that $F(x,\tilde \Phi(x),\tilde \Psi \circ G(x))=0$.
\end{thm}

Before proving this result, let us show an immediate application to the so-called Right-Left equivalence relation of Nash function maps, in the case $\K=\R$. We say that the germs at the origin of the Nash maps $f:\R^n\rightarrow \R^m$ and $g:\R^n\rightarrow \R^m$ are formally Right-Left (resp. Nash Right-Left) equivalent if there exist some smooth (resp. Nash) 
diffeomorphisms $\phi:\R^n\rightarrow\R^n$ and $\psi:\R^m\rightarrow\R^m$,
such that the Taylor expansions of $g\circ \phi$ and $\psi\circ f$ at the origin coincide (resp. $g\circ \phi=\psi\circ f$). One recover the following (originally proved in \cite{S3}):
\begin{cor}
Let $f:\R^n\rightarrow \R^m$ and $g:\R^n\rightarrow \R^m$ be two Nash germs at the origin. Then, $f$ and $g$ are 
formally Right-Left equivalent if and only if they are Nash Right-Left equivalent.
\end{cor}
\begin{proof}
Assume $f$ and $g$ are formally Right-Left equivalent, so that there exist smooth diffeomorphisms $\phi$ and $\psi$ such that $g\circ \phi=\psi\circ f$ at the level of Taylor expansions. Denoting $F(y,z)=g(y)-z$, we obtain that the Taylor expansions of $\phi$ and $\psi$ are formal solutions of
$$F(\phi(x),\psi\circ f(x))=0.$$
Using Theorem \ref{thmlocal}, it is possible to approximate $\phi$ and $\psi$ by  Nash map germs $\tilde\phi$ and
$\tilde\psi$, which remain diffeomorphisms by (close enough) approximation, such that $$F(\tilde\phi(x),\tilde\psi\circ f(x))=0.$$ Hence $g\circ \tilde\phi=\tilde\psi\circ f$, so that $f$ and $g$ are Nash Right-Left equivalent. The converse implication is immediate.
\end{proof}

\begin{rmk} The result is no longer true if we replace algebraic power series with analytic power series. A counter-example in the real case is exhibited in \cite{S3}, based on a famous example of Osgood and Gabrielov \cite{G}. More precisely, there exist two germs of analytic maps $f,g:(\R^2,0) \to (\R^4,0)$ which are Left equivalent as smooth germs, but not Left equivalent as analytic germs. 
\end{rmk}

Let us come back now to the proof of Theorem \ref{thmlocal} which is a consequence of the nested version of Artin approximation in the local case (see \cite{T}). 

\begin{proof}[Proof of Theorem \ref{thmlocal}] Write $F=(F_1,\ldots,F_q)$. We translate the equation 
$$F(x,\Phi(x), \Psi\circ G(x))=0$$
into the conditions that 
$$
\left\{\begin{array}{lcl}
\forall k\in \{1,\ldots, q\}&& F_k(x,\Phi(x), \Psi(t))=0\\
\forall j \in \{1,\ldots,l \}&&t_j=G_j(x).
\end{array}\right.
$$
Since $G(0)=0$, the set 
$$\{x_1,\ldots,x_m, t_1-G_1(x),\ldots,t_l-G_l(x)\}$$ 
may serve as a new set of variables for $\K[[x,t]]$. In particular, considering the $F_k$'s as formal power series, we obtain, for any $k\in \{1,\ldots, p\}$, a unique decomposition
$$F_k(x,\Phi(x),\Psi(t))=\sum_{(a,b)\in (\N\cup \{0\})^m \times (\N\cup \{0\})^l} c_{a,b}x^a(t-G(x))^b$$
with $c_{a,b}\in \K$, $x^a=\prod_{i=1}^m x_i^{a_i}$ and $(t-G(x))^b=\prod_{j=1}^l (t_j-G_j(x))^{b_j}$. Moreover the coefficients $c_{a,0}$ must vanish since $\Phi$ and $\Psi$ are solutions of
$$F(x,\Phi(x), \Psi\circ G(x))=0.$$
So we can rewrite these decompositions of the $F_k$'s as finite sums
$$F_k(x,\Phi(x), \Psi(t))=\sum_{j=1}^l \xi_{k,j}(x,t)(t_j-G_j(x)) $$
for some $\xi_{k,j}\in \K[[x,t]]$. Consider the system of equations in $\K_{\alg}[[x,t]]$ with variables 
$$\{U_i,V_r,W_{k,j} \}_{(i,j,k,r)\in \{1,\ldots,n\}\times \{1,\ldots,l\}\times \{1,\ldots,q\}\times \{1,\ldots,p\}}$$
defined by:

$$
\left\{
\begin{array}{l}
F_1(x,U_1,\ldots,U_n,V_1,\ldots,V_p)-\sum_{j=1}^l W_{1,j}(t_j-G_j(x)) =0\\
\quad \quad \quad \quad \quad \quad \quad \quad \vdots\\
F_q(x,U_1,\ldots,U_n,V_1,\ldots,V_p)-\sum_{j=1}^l W_{q,j}(t_j-G_j(x)) =0.
\end{array}
\right.
$$
This system of equations in $\K_{\alg}[[x,t]]$ admits 
$$
\left\{
\begin{array}{l}
(U_1,\ldots,U_n)=\Phi(x)\\
(V_1,\ldots,V_p)=\Psi(t)\\
W_{k,j}=\xi_{k,j}(x,t),~~(j,k)\in  \{1,\ldots,l\}\times \{1,\ldots,q\}
\end{array}
\right.
$$
as a solution in $\K[[x,t]]$.
 By nested approximation theorem (\cite{T}), we can approximate this solution by a solution
 $$
\left\{
\begin{array}{l}
(U_1,\ldots,U_n)=\tilde \Phi(x,t)\\
(V_1,\ldots,V_p)=\tilde \Psi(t)\\
W_{k,j}=\tilde \xi_{k,j}(x,t),~~(j,k)\in  \{1,\ldots,l\}\times \{1,\ldots,q\}
\end{array}
\right.
$$
in $\K_{\alg}[[x,t]]$. As a consequence
$$F_k(x,\tilde \Phi(x,t),\tilde \Psi(t))=\sum_{j=1}^l \tilde \xi_{k,j}(x,t)(t_j-G_j(x)) $$
for any $k\in \{1,\ldots,q\}$, and therefore 
$$F(x,\tilde \Phi(x,G(x)),\tilde \Psi (G(x)))=0.$$
It remains to say that $\tilde {\tilde {\Phi}}(x)=\tilde \Phi(x,G(x))$ approximate $\tilde\Phi(x)$ in 
$\K_{\alg}[[x]]^n$ since  $\tilde \Phi(x,t)$ approximate $\tilde\Phi(x)$ in 
$\K_{\alg}[[x,t]]^n$, and therefore $\tilde {\tilde {\Phi}}$ and $\tilde \Psi$ give the desired solutions.
\end{proof}

\begin{rmk}\label{ChangeAndNested}
If the nested Artin approximation is the key tool for proving Theorem \ref{thmlocal}, conversely one may also note that the nested Artin approximation for two sets of variables $x_1$ and $x_2$ could be seen as a consequence of Theorem \ref{thmlocal}. 
Indeed, let $F\in \K_{\alg}[[x_1,x_2,y,z]]^q$ be such that $F(0)=0$. Assume that there exists $\Psi\in \K[[x_1]]$ and
$\Phi\in \K[[x_1,x_2]]$ such that $\Phi(0)=\Psi(0)=0$ and $$F(x_1,x_2,\Psi(x_1), \Phi(x_1,x_2))=0.$$ 
This last equation is equivalent to 
$$F(x_1,x_2,\Psi_2\circ G(x_1,x_2), \Phi(x_1,x_2))=0.$$
 where $G(x_1,x_2)=(x_1,0)$ and $\Psi_2$ is just $\Psi$ viewed in $\K[[x_1,x_2]]$. 
By Theorem \ref{thmlocal}, one can approximate $\Psi_2\in \K[[x_1,x_2]]$ and
$\Phi\in \K[[x_1,x_2]]$ respectively by $\tilde\Psi_2\in \K_{\alg}[[x_1,x_2]]$ and
$\tilde\Phi\in \K_{\alg}[[x_1,x_2]]$ such that 
$$F(x_1,x_2,\tilde\Psi_2\circ G(x_1,x_2), \tilde\Phi(x_1,x_2)))=0,$$
or in other words
$$F(x_1,x_2,\tilde\Psi_2(x_1,0), \tilde\Phi(x_1,x_2)))=0.$$
Since $\tilde\Psi_2(x_1,x_2)$ approximate $\Psi(x_1)$ in $\K[[x_1,x_2]]$, then 
$\tilde\Psi_2(x_1,0)$ approximate also $\Psi(x_1)$.
\par
In other words, nested Artin approximation for a set of two variables and Artin approximation compatible with a change on a set of two variables are closely related notions. 
\end{rmk}

One may readily generalize Theorem \ref{thmlocal} to a change on several sets of variables :
\begin{thm}\label{genthmlocal} Let $m,m_i,l_i,q\in \N$ and consider some sets of variables $x=(x_1,\ldots,x_m)$, $y_i=(y_{i,1},\ldots,y_{i,m_i})$, for $i\in\{0,\ldots,n\}$. Let $F\in \K_{\alg}[[x,y_0,\ldots,y_n]]^q$ and let $G_i\in \K_{\alg}[[x]]^{l_i}$ be such that $F(0)=0$ and $G_i(0)=0$ for $i\in\{0,\ldots,n\}$.

Assume there exist some $n_i$-tuples of formal power series $\Psi_i$ such that, for each $i\in\{0,\ldots,n\}$, one has $\Psi_i(0)=0$ and $$F(x,\Psi_0(x), \Psi_1(G_1(x)), \Psi_2(G_1(x),G_2(x)),\ldots, \Psi_n(G_1(x),\ldots,G_n(x)))=0.$$ Then, one can approximate all the $\Psi_i$'s by algebraic $n_i$-tuples of power series $\tilde \Psi_i$
such that $$F(x,\tilde\Psi_0(x), \tilde\Psi_1(G_1(x)), \tilde\Psi_2(G_1(x),G_2(x)),\ldots, \tilde\Psi_n(G_1(x),\ldots,G_n(x)))=0.$$
\end{thm}
As it has already been noted in Remark \ref{ChangeAndNested}, this result is equivalent to nested Artin approximation.

\section{Global case}
The global versions of Artin approximation we deal with in this paper, are very similar in the real and complex setting. We give more details for the real case, with which the reader may be less familiar. The common idea is to translate the modification in the source variables of the solution in terms of inclusion of Nash sets. Then we use Cartan Theorem B in order to describe this inclusion by global equations, increasing the number of variables. Then we apply nested approximation to obtain the desired solution, in the same way as in the local case.
\subsection{Real case}
The interest for Nash functions and Nash manifolds in real geometry comes from the original work of Nash \cite{N}, whose purpose was to equip any analytic manifold with a real algebraic structure. The use of Nash functions is now classical, in an intermediate step between polynomials and analytic functions (cf. \cite{S1}). Concerning more specifically Artin approximation, the first global version has been used in \cite{CRS} in order to solve difficult questions on Nash functions, 
whereas its nested version has  been used in \cite{FS} in order to approximate an analytic resolution of the singularities of a Nash function by a Nash resolution. \par
In all the following, the approximations are considered with respect to the $C^\infty$ topology.\par
For the convenience of the reader, we recall Proposition 3.1 in \cite{FS} (in the particular case where $M_i=X_i$ are compact Nash manifolds, for $i=1,\ldots,m$) that will be the key to prove the main result of this section and which can be seen as  a global counterpart of the nested approximation theorem. For a Nash manifold $M$, we denote by $\mathcal O$, respectively $\mathcal M$, its sheaf of analytic, respectively Nash, function germs.

\begin{thm}\label{NestedGlobal} Let $M_1,...,M_m$ be compact Nash manifolds and 
$l_1,...,l_m, n_1,...,n_m\in\N$. 
Let $F_i\in \mathcal N(M_1\times\cdots\times M_i\times\R^{l_1}\times\cdots\times\R^{l_i})^{n_i}$ and $f_i\in\mathcal O(M_1
\times\cdots\times M_i)^{l_i}$, for $i=1,...,m$, be such that 
$$F_i(x_1,...,x_i,f_1(x_1),...,f_i(x_1,...,x_i))=0,$$ as 
elements of $\mathcal O(M_1\times\cdots\times M_i)^{n_i}$.
Then there exist $\tilde f_i\in \mathcal N(M_1\times\cdots\times M_i)^{l_i}$, for $i=1,...,m$, close to $f_i$ in the $C^\infty$ topology, such that 
$$F_i(x_1,...,x_i,\tilde f_1(x_1),...,\tilde f_i(x_1,...,x_i))=0$$ 
in $\mathcal N(M_1\times
\cdots\times M_i)^{n_i}$. 
\end{thm}

Next result can be seen as an Artin approximation compatible with a change of variables, or in other words, as a global  
counterpart in the real setting of Theorem \ref{thmlocal}.

\begin{thm}\label{thmglobalR} Let $M, N, L$ and $P$ be compact Nash manifolds, $q\in N$ be an integer and $F:M\times N \times P \to \R^q$ and $G:M \to L$ be Nash maps.

Assume there exist analytic maps $\Phi: M\to N$ and $\Psi:L\to P$ such that $F(x,\Phi (x), \Psi \circ G(x))=0$ for any $x\in M$. Then there exist Nash maps $\tilde \Phi :M\to N$ and $\tilde \Psi:L \to P$, close to $\Phi$ and $\Psi$, such that $F(x,\tilde \Phi (x), \tilde \Psi \circ G(x))=0$ for any $x\in M$.
\end{thm}

Before entering into the details of the proof, we state an classical auxiliary lemma.

\begin{lemma}\label{lem1} Let $M$ be a Nash manifold, let $X\subset M$ be a Nash subset and take $x\in X$. Denote by $I\subset \mathcal N_{M,x}$ the ideal of Nash germs whose complexification vanishes on the complexification $X_{\mathbb C}$ of $X$ in a neighbourhood of $x$. Then $I\mathcal O_{M,x}$ is equal to the ideal of analytic germs whose complexification vanishes on $X_{\mathbb C}$ in a neighbourhood of $x$.
\end{lemma}

\begin{proof} The non obvious inclusion follows from the local complex analytic Ruckert's Nullstellensatz \cite{GN} and by faithfully flatness of $\mathcal O_{M,x}$ over $\mathcal N_{M,x}$. Actually, a germs at $x$ whose complexification vanishes on $X_{\mathbb C}$ belongs to the radical of $I\mathcal O_{M,x}$ by the Nullstellensatz. Moreover $I\mathcal O_{M,x}$ remains radical, by faithfully flatness, since so is $I$.
\end{proof}

\begin{proof}[Proof of Theorem \ref{thmglobalR}] We are going to translate the condition that $\Phi$ satisfies 
$$F(x,\Phi (x), \Psi \circ G(x))=0$$ 
for any $x\in M$ in such a way to be able to use nested approximation. The idea is to note that this condition is equivalent to the fact that the image of the graph of $G$ by the map 
$$({\rm Id},\Phi) \times \Psi:M\times L \to M\times N \times P$$
$$(x,y)\mapsto (x,\Phi(x),\Psi (y))$$ 
is included in the set $F^{-1}(0)$.

Assume $M, N, L$ and $P$ are respectively embedded in $\R^m, \R^n, \R^l$ and $\R^p$ as closed Nash manifolds, and let $h_N:\R^n\to \R$ and $h_P:\R^p\to \R$ be some global equations for $N$ and $P$ respectively (they exist by \cite{S1}).
Let $\Gamma \subset M\times L$ denote the graph of $G$, and let $\mathcal I$ be the sheaf of $\mathcal N_{M\times L}$-ideals of germs whose complexifications are vanishing on the complexification $\Gamma^{\mathbb C}$ of $\Gamma$. Then, we know by \cite{CRS} that there exist global generators $g_1,\ldots,g_s$ of $\mathcal I$. 

We want to approximate the analytic maps $\Phi$ and $\Psi$ keeping the property that $$(({\rm Id},\Phi) \times \Psi)(\Gamma) \subset F^{-1}(0).$$
By Lemma \ref{lem1} applied to $\Gamma \subset M\times L$, we can use Cartan Theorem B to deduce the surjectivity of the homomorphism 
$$H^0(M\times L, \mathcal O_{M\times L}^s) \to H^0(M\times L, \mathcal I  \mathcal O_{M\times L})$$ 
defined by $(\xi_1,\ldots,\xi_s)\mapsto \sum_{j=1}^s\xi_j g_j$ (note that this step is the major change with respect to what has been done in the lobal case).
In particular, considering the global cross-section 
$$F(({\rm Id},\Phi) \times \Psi)$$
of $\mathcal I \mathcal O_{M\times L}$, we obtain that there exist analytic functions $\xi_{k,1},\ldots,\xi_{k,s}$ on $M\times L$ such that
$$\forall k\in \{1,\ldots,q \},~ \forall (x,y)\in M\times L,\quad F_k(x,\Phi(x),\Psi(y))=\sum_{j=1}^s \xi_{k,j}(x,y)g_j(x,y).$$
Moreover $h_N(\Phi(x))=0$ for  $x\in M$ and $h_P(\Psi(y))=0$ for $y\in L$. Now consider the system of equations
$$\left\{\begin{array}{l}
F_1(x,U_1,\ldots,U_n,V_1,\ldots,V_p)-\sum_{j=1}^s W_{1,j}g_j(x,y)=0\\
\quad \quad \quad \quad \quad \quad \quad \quad \vdots\\
F_q(x,U_1,\ldots,U_n,V_1,\ldots,V_p)-\sum_{j=1}^s W_{q,j}g_j(x,y)=0\\
h_N(U_1,\ldots,U_n)=0\\
h_P(V_1,\ldots,V_p)=0
\end{array}\right.$$
with variables 
$$(U_i, V_r, W_{k,j})_{(i,j,k,r)\in \{1,\ldots,n\}\times \{1,\ldots,s\}\times \{1,\ldots,q\}\times \{1,\ldots,p\}}.$$
This system of equations involves Nash maps and functions, and it admits analytic solutions given by  
$$
\left\{
\begin{array}{l}
(U_1,\ldots,U_n)=\Phi(x)\\
(V_1,\ldots,V_p)=\Psi(y)\\
W_{k,j}=\xi_{k,j}(x,y)\quad {\rm for\, all}\; (j,k)\in  \{1,\ldots,s\}\times \{1,\ldots,q\}.
\end{array}
\right.
$$
By nested approximation in the Nash case (Theorem \ref{NestedGlobal} applied to $(x_1,x_2)=(y,x)$), we can approximate these analytic solutions by Nash solutions
 
 $$
\left\{
\begin{array}{l}
(U_1,\ldots,U_n)=\tilde \Phi(x,y)\\
(V_1,\ldots,V_p)=\tilde \Psi(y)\\
W_{k,j}=\tilde \xi_{k,j}(x,y),\quad {\rm for\, all}\; (j,k)\in  \{1,\ldots,s\}\times \{1,\ldots,q\}.
\end{array}
\right.
$$
In particular $\tilde \Phi$ and $\tilde \Psi$ take respectively values in $N$ and $P$, and moreover, evaluating at $y=G(x)$, we obtain for any $x\in M$ the equality
$$F(x,\tilde \Phi (x,G(x)),\tilde \Psi (G(x)))=0.$$
 We conclude by remarking that the Nash map $x\mapsto \tilde \Phi (x,G(x))$ still is an approximation of $\Phi$, and together with $\tilde \Psi$, they give the requested solution.
\end{proof}

\begin{rmk}\begin{enumerate}
\item Similarly to the local case, we deduce easily from Theorem \ref{thmglobalR} the main result in \cite{S3}, namely that Nash maps on a compact Nash manifold are Nash Right-Left equivalent if they are analytic Right-Left equivalent.
\item Theorem \ref{thmglobalR} can be refined to the case of germs of Nash maps, defined on germs of compact semialgebraic sets in Nash manifolds (cf. \cite{S3} for example).
\end{enumerate}
\end{rmk}

\subsection{Complex case}
In the complex setting, the result is very similar, providing to use the correct analogue of the global nested Artin approximation. By \cite{L}, we can approximate a holomorphic solution of a Nash algebraic equation on a compact polynomial polyhedron (namely the intersection of a finite number of sets of the form $|f|\leq 1$ for $f$ a complex polynomial; note that the ring of holomorphic functions on such a set is Noetherian by \cite{F}) by a Nash algebraic solution. Here by Nash algebraic, we follow the terminology of \cite{L} who defined a Nash algebraic map to be a homolorphic map between affine complex algebraic varieties whose components satisfy polynomial equations. 
Therefore we obtain next result by a proof completely similar to that of Theorem \ref{thmglobalR}:

\begin{thm}\label{thmglobalC} Let $M, N, L$ and $P$ be non singular affine complex algebraic varieties, and $F:M\times N \times P \to \C^q$ and $G:M \to L$ be regular maps.

Assume that there exist some holomorphic maps $\Phi: M\to N$ and $\Psi:L \to P$ such that $F(x,\Phi (x), \Psi \circ G(x))=0$ for any $x\in M$. Let $A\subset M$ and $B \subset L$ be compact polynomial polyhedra satisfying $G(A)\subset B$. Then $\Phi_{|A}$ and $\Psi_{|B}$ can be approximated by Nash algebraic maps 
$\tilde \Phi :A\to N$ and $\tilde \Psi:B\to P$ such that $F(x,\tilde \Phi (x),\tilde \Psi \circ G(x))=0$ for $x\in A$.
\end{thm}



\enddocument